\documentclass[a4paper,12pt,draft]{amsart}
\usepackage{amsmath, amssymb, amsthm, url}
\usepackage{braket}
\usepackage{easywncy}
\newcommand{\eel}{\mathcyr{l}}
\newcommand{\ees}{\mathcyr{s}}

\newcommand{\ZZ}{\mathbb{Z}}

\newcommand{\CCC}{\mathcal{C}}

\theoremstyle{plain}   

\newtheorem{thm}{Theorem}[section]
\newtheorem{theorem}[thm]{Theorem}

\newtheorem{lemma}[thm]{Lemma}

\newtheorem{proposition}[thm]{Propsition}

\theoremstyle{remark}  

\newtheorem{remark}[thm]{Remark}

\theoremstyle{definition}  

\allowdisplaybreaks[3]

\begin{document}


\author[Y. Numata]{NUMATA, Yasuhide}
 \address[Numata]{Department of Mathematical Sciences\\
 Shinshu University\\
 3-1-1 Asahi, Matsumoto-shi, Nagano-ken, 390-8621, Japan.}
 \thanks{This work was supported by JSPS KAKENHI Grant Number JP25800009}
 \subjclass[2010]{18A99, 18A32, 18D99}
 \keywords{Quasi-schemoids; Association schemes; Hamming schemes; Morita equivalent}
\title{On functors between categories with colored morphisms}

\begin{abstract}
 In this paper, we consider categories with colored morphisms
 and functors such that
 morphisms assigned to  morphisms with a common color 
 have a common color.
 In this paper,
 we construct
 a morphism-colored functor 
 such that any morphism-colored functor from 
 a given small morphism-colored groupoid 
 to any discrete morphism-colored category
 factors through it.
 We also apply the main result to
 a schemoid constructed from a Hamming scheme.
\end{abstract}
\newcommand{\resp}{{\em resp.}\ }
\newcommand{\xgpd}{\widetilde{\operatorname{Gpd}}}
\newcommand{\Stab}{\operatorname{Stab}}
\newcommand{\obj}{\operatorname{Obj}}
\newcommand{\mor}{\operatorname{Mor}}
\newcommand{\Hom}{\operatorname{Hom}}
\newcommand{\target}{\operatorname{target}}
\newcommand{\source}{\operatorname{source}}
\newcommand{\id}{\operatorname{id}}
\newcommand{\Img}{\operatorname{Img}}
\newcommand{\FF}{\mathbb{F}}
\newcommand{\KK}{\mathbb{K}}
\newcommand{\uuu}{\operatorname{U}}
\newcommand{\Aut}{\operatorname{Aut}}
\newcommand{\Orb}{\operatorname{Orb}}
\newcommand{\sssim}{
\stackrel{1}{%
\sim%
}
}
\newcommand{\ssim}{
%
\stackrel{0}{\sim}
}
\newcommand{\ssimo}{
%
\stackrel{0'}{\sim}
}

\newcommand{\GGG}{\mathcal{G}}
\newcommand{\UUU}{\mathcal{U}}
\maketitle
\newcommand{\defit}[1]{\textit{#1}}

 \section{Introduction}
 In this paper, we consider categories with `colored' morphisms
 and functors $F$ such that
 the colors of $F(f)$ and $F(g)$ is the same
 if the colors of morphisms  $f$ and $g$ is the same.
 An example of such category is a schemoid
 introduced in Kuribayashi and Matsuo  \cite{MR3318282}.
 There exists a natural way to construct
 a schemoid from an association scheme.
 In this sense,
 a schemoid is a  generalization of
 an association scheme,
 introduced by Bose--Shimamoto \cite{MR0048772}.
 Kuribayashi and Matsuo also introduced 
 an analogue, called a schemoid algebra, of the Bose--Mesner algebra for an association scheme.
 The Bose--Mesner algebra is introduced in Bose--Mesner \cite{MR0102157},
 and plays an important role for algebraic study for association schemes.
 Kuribayashi and Matsuo studies the structure of
 schemoid algebras.
 Kuribayashi \cite{MR3306876} and 
 Kuribayashi--Momose \cite{arxiv:1507.01745}
 develop homotopy theory for schemoids.
 In  \cite{arxiv:1507.01745},  Kuribayashi and Momose
 shows that
 schemoids constructed from Hamming schemes over the field of order 2
 are Morita equivalent.
 The purpose of this paper is
 to generalize their methods to 
 small morphism-colored groupoids.

 This paper is organized as follows:
 In Section \ref{subsec:def}, we 
 consider a small morphism-colored groupoid $(\GGG,\eel)$
 such that 
  $\eel(f)=\eel(g)$ implies $\eel(f^{-1})=\eel(g^{-1})$, 
 and 
 define 
 a morphism-colored functor $\varpi(\GGG,\eel)$ 
 such that any morphism-colored functor 
 from $(\GGG,\eel)$ 
 to any discrete morphism-colored category $(\CCC,\id)$
 factors through $\varpi(\GGG,\eel)$.
 We show some lemmas in Section \ref{subsec:proof}.
 In Section \ref{sec:app},
 we apply our theorem to 
  schemoids constructed from Hamming schemes as an example.

  \section{Definition and Main results}
  \label{subsec:def}
  In this paper we call the pair $(\CCC,\eel)$
  a \defit{morphism-colored category}
  if the following condition holds:
  \begin{itemize}
   \item 
	 $\CCC$ is a category.
   \item 
	 $\eel$ assigns a `color' $\eel(f)$ to 
	 each morphism $f$ in $\CCC$.
   \item 
	 If $g,f_1,f_2\in\mor(\CCC)$ satisfies
	 $\eel(g)=\eel(f_1\circ f_2)$, then
	 there exist $g_1,g_2\in\mor(\CCC)$ such that
	\begin{align*}
	 &g=g_1\circ g_2,\\
	 &\eel(g_1)=\eel(f_1), \\
	 &\eel(g_2)=\eel(f_2). 
	\end{align*}
  \end{itemize}
 For morphisms $g,f_1,\ldots,f_l \in \mor(\CCC)$ 
 of a morphism-colored category $(\CCC,\eel)$ 
 satisfying $\eel(g)=\eel(f_1\circ\cdots\circ f_l)$,
 it is easy to show that
	there exist $g_1,\ldots,g_l\in\mor(\CCC)$ such that
	\begin{align*}
	 g&=g_1\circ\cdots\circ g_l,\\
	 \eel(g_i)&=\eel(f_i) \quad (i=1,\ldots,l).
	\end{align*}
 For a category $\CCC$,
 the pairs $(\CCC,\id)$ and $(\CCC,0)$ 
 are morphism-colored categories,
 where $\id(f)=f$ and $0(f)=0$ for each morphism $f$ in $\CCC$.
 In this paper,
 we call $(\CCC,\id)$  a \defit{discrete morphism-colored category}.
 We say that a morphism-colored category $(\CCC,\eel)$ is small
 if 
 $\CCC$ is a small category and $\eel$ is a map from $\mor(\GGG)$ to a set.
 We call a morphism-colored category $(\GGG,\eel)$ 
 a \defit{morphism-colored groupoid}
 if $\GGG$ is a groupoid.

 Let $(\CCC,\eel)$ and $(\CCC',\eel')$ 
 be morphism-colored categories.
 Let $F$ be a functor from $\CCC$ to $\CCC'$.
 We call a pair $(F,\gamma)$ a \defit{morphism-colored functor}
 from $(\CCC,\eel)$ to $(\CCC',\eel')$ 
 if $\gamma(\eel(f))=\eel'(F(f))$ for each morphism $f$ in $\CCC$.
 For a morphism-colored functor 
 $(F,\gamma)$ from $(\CCC,\eel)$ to $(\CCC',\eel')$ 
 and a morphism-colored functor $(F',\gamma')$ from $(\CCC',\eel')$ to $(\CCC'',\eel'')$, 
 we define the composition $(F',\gamma')\circ (F,\gamma)$ to be $(F'\circ F,\gamma'\circ \gamma)$.
 The composition 
$(F',\gamma')\circ (F,\gamma)$
of morphism-colored
 functors is a morphism-colored functor.

Let $(\GGG,\eel)$ be  a morphism-colored small groupoid.
Assume that 
$\eel(f^{-1})=\eel(g^{-1})$
for $f,g\in\mor(\GGG)$ with $\eel(f)=\eel(g)$.
  The main purpose of this section
  is to construct a groupoid $\UUU(\GGG,\eel)$
  and a morphism-colored functor $\varpi(\GGG,\eel)$ 
  from $(\GGG,\eel)$
  to the discrete morphism-colored category $(\UUU(\GGG,\eel),\id)$ 
  such that any morphism-colored functor $(F,\gamma)$ 
  from $(\GGG,\eel)$ 
   to any discrete morphism-colored category $(\CCC,\id)$
  factors through $\varpi(\GGG,\eel)$.

  Fix a morphism-colored small groupoid $(\GGG,\eel)$ such that 
  $\eel$ is a map from $\mor(\GGG)$ to a set $I$
  satisfying $\eel(f^{-1})=\eel(g^{-1})$
  for all pairs $f,g\in\mor(\GGG)$ such that $\eel(f)=\eel(g)$.
  Let $I_1$ be the image of $\eel$ and $\eel_1$
  the surjection from $\mor(\GGG)$ to $I_1$ defined by $\eel_1(f)=\eel(f)$.
  We also define $I_0$ to be the set $\Set{\eel(\id_x)| x\in\obj(\GGG)}$.
  Let $\eel_0$ be the surjection from $\obj(\GGG)$ to $I_0$
  defined by $\eel_0(x)=\eel(\id_x)$.

  First we define the set $\bar I_1$ of morphisms in $\UUU(\GGG,\eel)$.
  We define the binary relation $\sssim$ on $I_1$ in the following manner:
  \begin{align*}
   \eel(f_1\circ \cdots \circ f_l)\sssim \eel(g_1\circ \cdots \circ g_l)
  \end{align*}
  for morphisms $f_1,\ldots,f_l,g_1,\ldots,g_l$
  such that　
  $\eel(f_1)=\eel(g_1),\ldots,\eel(f_l)=\eel(g_l)$.
  The relation $\sssim$ is an 
  equivalence relation on $I_1$ (Proposition \ref{prop35}).
    Let  $\ees_1$ be the canonical surjection from $I_1$ to 
    the set $\bar I_1$ of equivalence classes  with respect to $\sssim$.    
    For $a\in I_1$,
    $\ees_1(a)$ is the equivalence class containing $a$.
    The composition $\ees_1\circ\eel_1$ is
    a surjection from $\mor(\GGG)$ to $\bar I_1$.

 Next we define the set $\bar I_0$ of objects of $\UUU(\GGG,\eel)$.
 We define the binary relation $\ssim$ on $I_0$ in the following manner:
 \begin{align*}
  \eel_0(\source(f))\ssim \eel_0(\source(g)) 
 \end{align*}
for $f,g\in\mor(\GGG)$ such that
$(\ees_1\circ\eel_1)(f)=(\ees_1\circ\eel_1)(g)$.
 The relation $\ssim$ is an equivalence relation on $I_0$ (Proposition \ref{prop:eq}). 
 Let $\ees_0$ be 
 the canonical surjection from $I_0$ to
 the set $\bar I_0$ of 
 equivalence classes with respect to $\ssim$.
 For $a\in I_0$,
 $\ees_0(a)$ is the equivalence class containing $a$.
 The composition $\ees_0\circ\eel_0$
 is a surjection from $\obj(\GGG)$ 
 to $\bar I_0$.

Since Propositions \ref{prop39} and \ref{prop310}
imply the well-definedness of the composition,
we have the following:
\begin{theorem}
 We can define the small groupoid $\UUU(\GGG,\eel)$ in the following manner:
 \begin{align*}
  \obj(\UUU(\GGG,\eel))&=\bar I_0,\\
  \mor(\UUU(\GGG,\eel))&=\bar I_1.
 \end{align*}
 The source of $(\ees_1\circ\eel_1)(f)\in \bar I_1$ is
 $(\ees_0\circ\eel_0)(\source(f))$
 and 
 the target of $(\ees_1\circ\eel_1)(f)\in \bar I_1$ is
 $(\ees_0\circ\eel_0)(\target(f))$.
 For a composable pair $(f,g)$ of morphisms in $\GGG$, 
 we define $(\ees_1\circ\eel_1)(f)\circ (\ees_1\circ\eel_1)(g)=(\ees_1\circ\eel_1)(f\circ g)$. 
\end{theorem}

 There exists a functor $\bar\eel$ from $\GGG$ to $\UUU(\GGG,\eel)$
 defined by
  $\bar\eel(x)=(\ees_0\circ\eel_0)(x)$ for $x\in \obj(\GGG)$,
  $\bar\eel(f)=(\ees_1\circ\eel_1)(f) \in \bar I_1$ for $f\in\mor(\GGG)$.
  Moreover  $(\bar\eel,\ees_1)$ is
  a morphism-colored functor
  from $(\GGG,\eel)$ to 
  the discrete morphism-colored category
  $(\UUU(\GGG,\eel),\id)$.
  Define $\varpi(\GGG,\eel)$ to be
  the morphism-colored functor $(\bar\eel,\ees_1)$.
  We show that
   each morphism-colored functor $(F,\gamma)$ 
   from $(\GGG,\eel)$ to
   a discrete morphism-colored category $(\CCC,\id)$
   factors through $\varpi(\GGG,\eel)$.
  \begin{theorem}
   For a morphism-colored functor $(F,\gamma)$ 
   from $(\GGG,\eel)$ to a discrete morphism-colored category
   $(\CCC,\id)$,
   there uniquely exists 
   a morphism-colored functor $(\check F,\check \gamma)$
   such that $(\check F,\check \gamma)\circ \varpi(\GGG,\eel)= (F,\gamma)$.
  \end{theorem}
  \begin{proof}
   Consider the morphism-colored functor $\varpi(\GGG,\eel)=(\bar\eel,\ees_1)$
   and a morphism-colored functor $(F,\gamma)$ from $(\GGG,\eel)$ to $(\CCC,\id)$. 
   In this case, we have $F(f)=F(g)$ for $f,g\in \mor(\GGG)$
   satisfying $\eel(f)=\eel(g)$.

   First we prove 
   the well-definedness of the correspondence  $\check F$
   from $\mor(\UUU(\GGG,\eel))$ 
   to  $\mor(\CCC)$ 
   such that
   $\check F((\ees_1\circ\eel_1)(f))=F(f)$ for $f\in \mor(\GGG)$.
   Let  $f,g\in \mor(\GGG)$ satisfy
   $(\ees_1\circ\eel_1)(f)=(\ees_1\circ\eel_1)(g)$.
   In this case,
   there exist $f_1,\ldots,f_l$ and $g_1,\ldots,g_l$
   such that 
   \begin{align*}
    f&=f_1\circ\cdots\circ f_l\\ 
    g&=g_1\circ\cdots\circ g_l\\ 
    \eel(f_i)&=\eel(g_i) \quad(i=1,\ldots,l).
   \end{align*}
   Since 
   $\eel(f_i)=\eel(g_i)$ implies $F(f_i)=F(g_i)$,
   we obtain $F(f)=F(g)$.

   Next we prove 
   the well-definedness of the  correspondence  $\check F$
   from $\obj(\UUU(\GGG,\eel))$ 
   to  $\obj(\CCC)$ 
   such that
   $\check F((\ees_0\circ\eel_0)(x))=F(x)$ for $x\in \obj(\GGG)$.
   Let $x,y\in \obj(\GGG)$ satisfy
   $(\ees_0\circ\eel_0)(x)=(\ees_0\circ\eel_0)(y)$.
   In this case,
   there exist $f,g\in \mor(\GGG)$
   such that
   \begin{align*}
    \eel_0(x)&=\eel_0(\source(f))\\
    \eel_0(y)&=\eel_0(\source(g))\\
    (\ees_1\circ \eel_1)(f)&=(\ees_1\circ \eel_1)(g).
   \end{align*}
   Since $\eel(\id_x)=\eel_0(x)=\eel_0(\source(f))=\eel(\id_{\source(f)})$,
   we have $F(\id_x)=F(\id_{\source(f)})$ which implies
   $F(x)=F(\source(f))=\source(F(f))$.
   Similarly we obtain $F(y)=\source(F(g))$.
   Since $(\ees_1\circ \eel_1)(f)=(\ees_1\circ \eel_1)(g)$,
   we have $F(f)=F(g)$.
   Hence 
 \begin{align*}
  F(x)=\source(F(f))=\source(F(g))=F(y). 
 \end{align*}

   Next we prove 
   the well-definedness of the  correspondence $\check \gamma$
   from $\mor(\UUU(\GGG,\eel))$ 
   to  $\mor(\CCC)$ 
   such that
   $\check\gamma(\ees_1(a))=\gamma(a)$ for $a\in I_1$.
   Let $a,b\in I_1$ satisfy
   $\ees_1(a)=\ees_1(b)$.
   In this case,
   there exist $f,g\in \mor(\GGG)$ such that
   $a=\eel(f)$ and $b=\eel(g)$.
   Since $\ees_1\circ\eel(f)=\ees_1(a)=\ees_1(b)=\ees_1\circ\eel(g)$, 
   we have $F(f)=F(g)$.
   Since  $(F,\gamma)$ is
    a morphism-colored functor
   from $(\GGG,\eel)$ to $(\CCC,\id)$,
   we have
   $\gamma(\eel(f))=F(f)$ and $\gamma(\eel(g))=F(g)$,
   which imply $\gamma(\eel(f))=\gamma(\eel(g))$.
   Hence $\gamma(a)=\gamma(b)$.

   It is easy to show that $(\check F,\check\gamma)$ is a morphism-colored
   functor,
   and that $(\check F,\check\gamma)\circ \varpi(\GGG,\eel)= (F,\gamma)$.
   Moreover, 
   the surjectivity of $\varpi(\GGG,\eel)$ implies the uniqueness of 
   $(\check F,\check\gamma)$.
  \end{proof}

\begin{remark}
 If $\bar I_0$ consists of one point,
 then $\bar I_1$ is a group.
 The group $\bar I_1$ is 
 isomorphic 
 to  the group 
 generated by $I_1$ with the following fundamental relations:
 \begin{align*}
   \eel(f\circ g) = \eel(f) \cdot \eel(g) 
 \end{align*}
for a composable pair $(f,g)$ of morphisms in $\GGG$.
\end{remark}

\begin{remark}
Consider 
another groupoid $(\GGG',\eel')$ such that
$\eel'(f)=\eel'(g)$ implies
$\eel'(f^{-1})=\eel'(g^{-1})$. 
In this case, we can also consider $\UUU(\GGG',\eel')$.
Let $\ees'_0$ and $\ees'_1$ be canonical surjections for $\UUU(\GGG',\eel')$. 
For a morphism-colored functor $(F,\gamma)$ 
from $(\GGG,\eel)$ to $(\GGG',\eel')$,
we can define a functor $\bar\gamma$ from $\UUU(\GGG,\eel)$ to $\UUU(\GGG',\eel')$
in the following manner:
$\bar\gamma(\ees_0(a))=\ees'_0(\gamma(a))$ for $\ees_0(a) \in \obj(\UUU(\GGG,\eel))$,
and
$\bar\gamma(\ees_1(b))=\ees'_1(\gamma(b))$ for $\ees_1(b) \in
 \mor(\UUU(\GGG,\eel))$.
In this sense, the operation $\UUU$ is functorial.
\end{remark}

\section{Proofs of lemmas}
\label{subsec:proof}

First we prove a technical lemma for an arbitrary morphism-colored category.
 \begin{lemma}
  \label{lemma:move}
  Let $(\GGG,\eel)$ be a morphism-colored category.
  For  $f,g\in\mor(\GGG)$ satisfying
  $\eel(f)=\eel(g)$, the following hold:
  \begin{enumerate}
   \item   For $t\in\mor(\GGG)$ satisfying $\source(f)=\source(t)$,
	   there exists $q\in\mor(\GGG)$
	   such that  
	   $\eel(t)=\eel(q)$
	   and $\source(g)=\source(q)$.
   \item   For $t'\in\mor(\GGG)$ satisfying $\target(f)=\target(t')$,
	   there exists $q'\in\mor(\GGG)$
	   such that  
	   $\eel(t')=\eel(q')$
	   and $\target(g)=\target(q')$.
  \end{enumerate}
 \end{lemma}
 \begin{proof}
  First we prove the first proposition.
  Let $t'=f\circ t^{-1}$.
  Since $f = t'\circ t$ and $\eel(f)=\eel(g)$,
  there exist $q',q\in\mor(\GGG)$ such that
  $g=q'\circ q$, $\eel(q')=\eel(t')$ and $\eel(q)=\eel(t)$.
  The morphism $q$ satisfies $\source(q)=\source(g)$.
  Similarly we can prove the second proposition.
 \end{proof}
 We also prove a technical lemma for an arbitrary morphism-colored
 groupoid
 such that inverses of morphisms with a common color have a common
 color.
 \begin{lemma}
  \label{lemma:move2}
  Let a morphism-colored groupoid  $(\GGG,\eel)$
    satisfy $\eel(f^{-1})=\eel(g^{-1})$
  for all pairs $f,g\in\mor(\GGG)$ such that $\eel(f)=\eel(g)$.
  For  $f,g\in\mor(\GGG)$ satisfying
  $\eel(f)=\eel(g)$, the following hold:
  \begin{enumerate}
   \item   For $t\in\mor(\GGG)$ satisfying $\source(f)=\target(t)$,
	   there exists $q\in\mor(\GGG)$
	   such that  
	   $\eel(t)=\eel(q)$
	   and $\source(g)=\target(q)$.
   \item   For $t'\in\mor(\GGG)$ satisfying $\target(f)=\source(t')$,
	   there exists $q'\in\mor(\GGG)$
	   such that  
	   $\eel(t')=\eel(q')$
	   and $\target(g')=\source(q')$.
  \end{enumerate}
 \end{lemma}
 \begin{proof} 
  First we show the first proposition.
  Since 
  \begin{align*}
   \source(t^{-1})=\target(t)=\source(f),
  \end{align*}
  it follows from Lemma \ref{lemma:move} that
  there exists $q'$ such that 
	   $\eel(t^{-1})=\eel(q')$
	   and $\source(g)=\source(q)$.
  Let $q=(q')^{-1}$.
  It is easy to show that
	   $\eel(t)=\eel(q)$
	   and that $\source(g)=\target(q)$.
  Similarly we can show the second proposition.
 \end{proof}

 Now we show some propositions used in Section \ref{subsec:def}.
 Let us keep the same notation as in Section \ref{subsec:def}.
 First we show that $\sssim$ is an equivalence relation.
    \begin{proposition}
     \label{prop35}
     The relation $\sssim$ is an equivalence relation on $I_1$. 
    \end{proposition}
    \begin{proof}
     It is easy to show the reflexivity and the symmetry of the relation. 
     We show the transitivity of the relation. 
     Let $a\sssim b$.
     There exist
     $f_1,\ldots,f_l$ and  $p_1,\ldots, p_l \in \mor(\GGG)$
     such that
  \begin{align*}
   a&=\eel(f_1\circ \cdots \circ f_l),\\
   b&=\eel(p_1\circ \cdots \circ p_l),\\
   \eel(f_i)&=\eel(p_i) \quad (i=1,\ldots,l).
  \end{align*}
    Let $f=f_1\circ \cdots \circ f_l$,
  and $p=p_1\circ \cdots \circ p_l$.
  Let $b\sssim c$.
  There exist
  $g_1,\ldots,g_m$ and  $h_1,\ldots,h_m \in \mor(\GGG)$
  such that
  \begin{align*}
   b&=\eel(g_1\circ \cdots \circ g_m),\\
   c&=\eel(h_1\circ \cdots \circ h_m),\\
   \eel(g_i)&=\eel(h_i) \quad (i=1,\ldots,m).
  \end{align*}
  Let $g=g_1\circ \cdots \circ g_m$,
  and $h=h_1\circ \cdots \circ h_m$.
  Since $\eel(p)=\eel(g)$,
  there exist $q_1,\ldots, q_m\in\mor(\GGG)$ such that
  $ \eel(q_i)=\eel(g_i)$ for all $i$ and $p=q_1\circ\cdots\circ q_m$.
  Let $t=(q_1\circ \cdots \circ q_{m-1})^{-1}\circ p_1$. 
  Since $p_1=q_1\circ \cdots \circ q_{m-1}\circ t$ and $\eel(p_1)=\eel(f_1)$,
  there exist $h'_1,\ldots,h'_{m-1},t'$ such that 
  $ \eel(h'_i)=\eel(q_i)$ for all $i$, 
  $\eel(t)=\eel(t')$ and $f_1=h'_1\circ\cdots\circ h'_{m-1}\circ t'$.
  On the other hand,
  we have
  \begin{align*}
   t\circ p_2\circ\cdots\circ p_l
   &= (q_1\circ \cdots \circ q_{m-1})^{-1}\circ p_1 \circ p_2\circ\cdots\circ p_l\\
   &= q_m\circ (q_1\circ \cdots \circ q_{m})^{-1}\circ p_1 \circ p_2\circ\cdots\circ p_l\\
   &= q_m\circ p^{-1}\circ p\\
   &= q_m.
  \end{align*}
  Since $q_m=t\circ p_2\circ\cdots\circ p_l$ and $\eel(q_m)=\eel(h_m)$,
  there exist $f'_2,\ldots,f'_{l},t''$ such that 
  $ \eel(f'_i)=\eel(p_i)$ for all $i$, 
  $\eel(t)=\eel(t'')$ and $h_m=t''\circ f'_2\circ\cdots\circ f'_{l}$.
  Hence we have the following decompositions for $f$ and $h$:
  \begin{align*}
   f&=h'_1\circ\cdots\circ h'_{m-1}\circ t'\circ f_2\circ \cdots \circ f_l,\\
   h&=h_1\circ\cdots\circ h_{m-1}\circ t''\circ f'_2\circ\cdots\circ f'_{l}.
  \end{align*}
  Therefore $\eel(f)=a\sssim \eel(h)=c$.
    \end{proof}

    Next we show that $(\GGG,\ees_1\circ\eel_1)$ is 
    a morphism-colored category.
  \begin{lemma}
   \label{lem:morphism-colored}
   If $g,f_1,f_2\in\mor(\GGG)$ satisfies
   $(\ees_1\circ\eel_1)(g)=(\ees_1\circ\eel_1)(f_1\circ f_2)$, then
   there exist $g_1,g_2\in\mor(\GGG)$ such that
  \begin{align*}
   g&=g_1\circ g_2,\\
   (\ees_1\circ\eel_1)(g_1)&=(\ees_1\circ\eel_1)(f_1),\\
   (\ees_1\circ\eel_1)(g_2)&=(\ees_1\circ\eel_1)(f_2).
  \end{align*}
  \end{lemma}
 \begin{proof}
  Let $f=f_1\circ f_2$ and $g$ satisfy
  $a=\eel(f)$ and $b=\eel(g)$.
  Assume that $\ees_1(a)=\ees_1(b)$.
  By definition,
  there exist $t'_1,\ldots,t'_n$ and  $q'_1,\ldots,q'_n \in \mor(\GGG)$
  such that
  \begin{align*}
   a&=\eel(t'_1\circ \cdots \circ t'_n),\\
   b&=\eel(q'_1\circ \cdots \circ q'_n),\\
   \eel(t'_i)&=\eel(q'_i) \quad (i=1,\ldots,n).
  \end{align*}
  Since $\eel(f)=a = \eel(t'_1\circ \cdots \circ t'_n)$,
  there exist $t_1,\ldots,t_n$ such that
  \begin{align*}
   f&=t_1\circ \cdots \circ t_n,\\
   \eel(t_i)&=\eel(t'_i) \quad (i=1,\ldots,n).
  \end{align*}
  Since $\eel(g)=b=\eel(q'_1\circ \cdots \circ q'_n)$,
  there exist $q_1,\ldots,q_n$ such that
  \begin{align*}
   g&=q_1\circ \cdots \circ q_n,\\
   \eel(q_i)&=\eel(q'_i)=\eel(t'_i)=\eel(t_i) \quad (i=1,\ldots,n).
  \end{align*}
  Let $t''=(t_1\circ\cdots \circ t_{n-1})^{-1}\circ f_1$.
  It follows that
 \begin{align*}
  t''\circ f_2 &=(t_1\circ\cdots \circ t_{n-1})^{-1}\circ f_1\circ f_2 \\
  &=(t_1\circ\cdots \circ t_{n-1})^{-1}\circ f\\
  &=(t_1\circ\cdots \circ t_{n-1})^{-1}\circ (t_1\circ\cdots \circ t_{n})=t_n.
 \end{align*} 
  Since  
  $t_n=t''\circ f_2$
  $\eel(t_n)=\eel(q_n)$,
  there exist  $q''\in\mor(\GGG)$ such that
  $q_n=q''\circ q'''$, $\eel(q'')=\eel(t'')$ and $\eel(f_2)=\eel(q''')$.
  Let $g_1=q_1\circ \cdots\circ q_{n-1}\circ q'' $
  and $g_2=q'''$.
  It follows that
  \begin{align*}
   g_1\circ g_2&= q_1\circ \cdots\circ q_{n-1}\circ q'' \circ q'''\\
   &=q_1\circ \cdots\circ q_{n-1}\circ q_n= g.
  \end{align*}
  Since $\eel(g_2)=\eel(q''')=\eel(f_2)$,
  we have $(\ees_1\circ\eel_1)(g_2)=(\ees_1\circ\eel_1)(f_2)$.
  Since $f_1=t_1\circ \cdots\circ t_{n-1}\circ t''$ and
  $g_1=q_1\circ \cdots\circ q_{n-1}\circ q''$,
  we have
  $(\ees_1\circ\eel_1)(f_1)=(\ees_1\circ\eel_1)(g_1)$.
 \end{proof}

Let  $\ssimo$ be the binary relation on $\obj(\GGG)$
defined by
$\source(f)\ssimo \source(g)$
for $f,g \in \mor(\GGG)$
such that $(\ees_1\circ\eel_1)(f)=(\ees_1\circ\eel_1)(g)$.
\begin{remark}
For a small morphism-colored groupoid $(\GGG,\eel)$
 such that $\eel(f)=\eel(g)$ implies $\eel(f^{-1})=\eel(g^{-1})$,
we have
 \begin{align*}
&  \Set{
(\source(f),\source(g))
| (\ees_1\circ\eel_1)(f)=(\ees_1\circ\eel_1)(g) }\\
&= \Set{
(\target(f),\target(g)) 
| (\ees_1\circ\eel_1)(f)=(\ees_1\circ\eel_1)(g) }.
 \end{align*}
\end{remark}
\begin{lemma}
\label{lem:equivformor}
The relation  $\ssimo$ 
is an equivalence relation on $\obj(\GGG)$.
\end{lemma}
\begin{proof}
It is easy to show 
 the reflexivity
and the symmetry
 of $\ssimo$.
We prove 
 the transitivity of the binary relation defined by $\ssimo$.
It follows from Lemma \ref{lem:morphism-colored} 
that $(\GGG,\ees_1\circ\eel_1)$ is a morphism-colored category.
Let $x,y,z\in\obj(\GGG)$ satisfy $x\ssimo y$ and $y\ssimo z$.
Since $x\ssimo y$, there exist morphisms $f$ and $f'\in\mor(\GGG)$ such that
$x=\source(f)$, $y=\source(f')$ and 
$(\ees_1\circ\eel_1)(f)=(\ees_1\circ\eel_1)(f')$.
Since $y\ssimo z$, there exist morphisms $g$ and $g'\in\mor(\GGG)$ such that
$y=\source(g)$, $y=\source(g')$ and 
$(\ees_1\circ\eel_1)(g)=(\ees_1\circ\eel_1)(g')$.
By applying
Lemma \ref{lemma:move}
to $f$, $f'$ and $g$, 
we have a morphism $g''\in\mor(\GGG)$
such that $\source(g'')=\source(f)=x$ and 
that $(\ees_1\circ\eel_1)(g'')=(\ees_1\circ\eel_1)(g)=(\ees_1\circ\eel_1)(g')$.
Hence $x\ssimo z$.
\end{proof}

By Lemma \ref{lem:equivformor},
the binary relation $\ssimo$
is an equivalence relation on $\obj(\GGG)$.
Hence we have the following proposition.
\begin{proposition}
\label{prop:eq}
 The relation $\ssim$ is an equivalence relation on $I_0$.
\end{proposition}


Finally we show the well-definedness of the composition of morphisms
of our groupoid $\UUU(\GGG,\eel)$.
\begin{proposition}
     \label{prop39}
   If composable pairs $(f,g)$ and $(f',g')$ of morphisms in $\GGG$
 satisfy
  $(\ees_1\circ\eel_1)(f)=(\ees_1\circ\eel_1)(f')$ and
  $(\ees_1\circ\eel_1)(g)=(\ees_1\circ\eel_1)(g')$,
 then 
 \begin{align*}
  (\ees_1\circ\eel_1)(f\circ g)=(\ees_1\circ\eel_1)(f'\circ g').
 \end{align*}
\end{proposition}
\begin{proof}
   Let composable pairs $(f,g)$ and $(f',g')$  of morphisms in $\GGG$
 satisfy
  $\eel(f)\sssim \eel(f')$ and
  $\eel(g)\sssim \eel(g')$.
  There exist $f_1,\ldots,f_l$, $f'_1,\ldots,f'_l$
  such that 
  $f=f_1\circ\cdots\circ f_l$,
  $f'=f'_1\circ\cdots\circ f'_l$, and
  $\eel(f_1)=\eel(f'_1),\ldots,\eel(f_l)=\eel(f'_l)$.
  There also exist $g_1,\ldots,g_m$, $g'_1,\ldots,g'_m$
  such that 
  $g=g_1\circ\cdots\circ g_m$,
  $g'=g'_1\circ\cdots\circ g'_m$, and
  $\eel(g_1)=\eel(g'_1),\ldots,\eel(g_m)=\eel(g'_m)$.
  Hence   $\eel(f\circ g)\sssim \eel(f'\circ g')$.
\end{proof}
\begin{proposition}
     \label{prop310}
 For $f,g\in \mor(\GGG)$ satisfying
 $(\ees_0\circ\eel_0)(\source(f))=(\ees_0\circ\eel_0)(\target(g))$,
 there exists
 a composable pair $(f',g')$ of morphisms in $\GGG$ such that
 $(\ees_1\circ\eel_1)(f)=(\ees_1\circ\eel_1)(f')$ and
 $(\ees_1\circ\eel_1)(g)=(\ees_1\circ\eel_1)(g')$.
\end{proposition}
  \begin{proof}
Since
   $(\ees_0\circ\eel_0)(\source(f))=(\ees_0\circ\eel_0)(\target(g))$,
there exist morphisms $h$ and $h'$ such that
$\source(h)=\source(f)$,
$\source(h')=\target(g)$,
and  $(\ees_1\circ\eel_1)(h)=(\ees_1\circ\eel_1)(h')$.
By applying 
Lemma \ref{lemma:move}
to $h$, $h'$ and $f$,
there exists $f'$ such that
 $(\ees_1\circ\eel_1)(f)=(\ees_1\circ\eel_1)(f')$ and
$\source(f')=\source(h')=\target(g)$.
Let $g'=g$.
Then
$(f',g')$ is  a composable pair such that
 $(\ees_1\circ\eel_1)(f)=(\ees_1\circ\eel_1)(f')$ and
 $(\ees_1\circ\eel_1)(g)=(\ees_1\circ\eel_1)(g')$.
 \end{proof}

 \section{Application}
 \label{sec:app}
 Consider the pair of a category $\CCC$ and
	$\eel$  which assigns a `color' $\eel(f)$ to 
	each morphism $f$ in $\CCC$.
	Assume that
 \begin{align*}
  \Set{ (f',g') \in\mor(\CCC)\times\mor(\CCC)
| 
   \begin{array}{c}
    \eel(f)=\eel(f'),\\ \eel(g)=\eel(g'),\\  
    f'\circ g' = h.
   \end{array}
 }  
 \end{align*}
 is a set for each $f,g,h\in\mor(\CCC)$.
 We write $N_h^{\eel(f),\eel(g)}$ to denote the set.
 The pair $(\CCC,\eel)$ satisfying the following condition
 is called a \defit{schemoid},
 at least when $\CCC$ is small:
 \begin{align*}
  \eel(h)=\eel(k) \implies
\text{there exists a bijection from $N_h^{\eel(f),\eel(g)}$ to $N_k^{\eel(f),\eel(g)}$}
 \end{align*}
 for each $f,g,h,k\in\mor(\CCC)$.
 It is easy to see that a schemoid is a morphism-colored category.
 For any category $\CCC$,
 the discrete morphism-colored category $(\CCC,\id)$
 is a schemoid.
 In this section,
 we consider 
 a schemoid which is equivalent to a Hamming scheme.
 We apply our main theorem to the schemoid.
 \begin{remark}
  A (small) schemoid in this paper is 
  called a quasi-schemoid in Kuribayashi--Matsuo \cite{MR3318282},
  and is defined to be a pair of small category 
  $\CCC$ and a partition $S$ of the morphisms with a condition.
  For a schemoid $(\CCC,\eel)$ in this paper,
  the pair $(\CCC,\Set{\eel^{-1}(\Set{i})|i\in I})$
  satisfies the condition in their paper.
 \end{remark}


Fix an integer $d$ such that $d>1$.
Let $G$ be the additive group $(\ZZ/n\ZZ)^d$.
 For $x=(x_1,\ldots,x_d)\in G$,
 we define $w(x)$ to be the Hamming weight $\#\Set{i|x_i\neq 0}$.
 We regard $w$ as a map from $G$ to $\Set{0,\ldots, d}$.
 Define $G/\!/G$ to be the action groupoid
 with respect to the natural action on $G$ of $G$, i.e.,
 the category whose set $G/\!/G$ of objects is $G$
 and whose set $\mor(G/\!/G)$ of morphisms is 
 $G\times G$.
 For $(g,x)\in \mor(G/\!/G)$, the source (\resp target) of $(g,x)$
 is $x$ (\resp $g+x$).
 For $x\in\obj(G/\!/G)$,
 the identity $\id_x$ on $x$ is $(0,x)$.
 The inverse of $(g,x)\in \mor((G/\!/G)$ is $(-g,g+x)$.
 Let $\pi$ be the projection 
 from $G/\!/G$ to $G$ defined by $\pi((g,x))=g$.
 The pair
 $(G/\!/G,\pi)$
 is a schemoid.
 Moreover
 $(G/\!/G,w\circ\pi)$
 is also a schemoid,
 which is equivalent to the Hamming scheme.
 Let 
\begin{align*}
  \GGG(d,n)&=(\ZZ/n\ZZ)^d/\!/(\ZZ/n\ZZ)^d, \\
  H(d,n)&=((\ZZ/n\ZZ)^d/\!/(\ZZ/n\ZZ)^d,w\circ\pi). 
\end{align*}
 Since 
$(w\circ\pi)((g,x)^{-1})=(w\circ\pi)((-g,g+x))=w(-g)=w(g)$,
  we have 
  \begin{align*}
   (w\circ\pi)((g,x)^{-1})=(w\circ\pi)((g',x')^{-1})
  \end{align*}
  for all pairs $(g,x),(g',x')\in\mor(\GGG)$ such that 
   $(w\circ\pi)((g,x))=(w\circ\pi)((g',x'))$.
 Now we apply our main theorem to the schemoid  
 $H(d,n)$.

 Let $\bar I_0=\obj(\UUU(H(d,n))$.
 Since $\id_x=(0,x)\in \mor(\GGG(d,n))$,
 we have $w\circ\pi(\id_x)=0$,
 which implies that $\bar I_0$ consists of one point.

 Let $\bar I_1=\mor(\UUU(H(n,d))$.
 Let $e_i$ be the $i$-th fundamental vector $(0,\ldots,0,1,0,\ldots,0)\in G$.
 It is obvious that $w\circ\pi((\sum_{i=1}^l e_i,x))=w(\sum_{i=1}^l e_i))=l$.
 Since 
\begin{align*}
 (\sum_{i=1}^l e_i,0)=(\sum_{i=1}^{l-1} e_i,e_l)\circ (e_l,0), 
\end{align*}
 $\bar I_0$ is a cyclic group generated by $\ees_1(1)=\ees_1\circ w\circ\pi((e_i,0))$.
 Let us use the symbol $+$ for the group operation of $\bar I_1$.

 If $n>2$, then
 \begin{align*}
  \ees_1(1)+\ees_1(1)
  &=\ees_1\circ w\circ\pi((e_1,e_1)\circ (e_1,0))\\
  &=\ees_1\circ w\circ\pi((2e_1,0))
  =\ees_1(1).
 \end{align*}
 Hence we have $\ees_1(1)=\ees_1(0)$,
 which implies that $\bar I_1$ consists of the unit.
 Therefore 
 the morphism-colored category $\UUU(H(d,n))$
 consists of one object $\ast$ and one identity $\id_\ast$,
 and
 $\varpi(H(d,n))$
 is 
 the trivial morphism-colored functor, i.e.,
 the morphism-colored functor $(\bar\eel,\ees)$
 defined by
  $\bar\eel(x)=\ast$ for $x\in \obj(\GGG(d,n))$,
  $\bar\eel(f)=\id_\ast$ for $f\in\mor(\GGG(d,n))$,
  $\ees(a)=\id_\ast$ for $a\in\Set{0,1,\ldots,d}$.
 In this case,
 any morphism-colored functor from $H(d,n)$ 
 to any discrete morphism-colored category
 factors through  the trivial morphism-colored functor.
 Therefore we have the following:
 \begin{proposition}
  Fix a category $\CCC$.
  Let $n,n',d,d'$ be positive integers.
  Assume that 
  $n, n'>2$.
  Each morphism-colored functor
  from the schemoid $H(d,n)$ 
  to the discrete morphism-colored category $(\CCC,\id)$
  factors through the trivial morphism-colored functor.
  Each morphism-colored functor from $H(d',n')$ to $(\CCC,\id)$
  also factors through  the trivial morphism-colored functor.
  Therefore
  we have a one-to-one correspondence between 
  morphism-colored functors from $H(n,d)$ to $(\CCC,\id)$
  and
  morphism-colored functors from $H(n',d')$ to $(\CCC,\id)$.
 \end{proposition}

 Consider the case where $n=2$.
 In this case,
 $\sum w(x^{(i)})$ is even
 for $x^{(1)},\ldots,x^{(k)}\in G$ satisfying $\sum_{i=1}^d x^{(i)}=0$.
 Hence the fundamental relation of $\bar I_1$ does not change the parity.
 We also have
 \begin{align*}
  \ees_1(1)+\ees_1(1)
  &=\ees_1\circ w\circ\pi((e_1,e_1)\circ (e_1,0))\\
  &=\ees_1\circ w\circ\pi((2e_1,0))
  =\ees_1\circ w\circ\pi((0,0))
  =\ees_1(0).
 \end{align*}
 Hence $\bar I_1$ is isomorphic to $\ZZ/2\ZZ$.
 Therefore
 $\UUU(H(d,2))$
 is the groupoid $\Set{ \ast }/\!/(\ZZ/2\ZZ)$
 with one object $\ast$ and two morphisms $0=\id_{\ast}$ and $1$;
 and
 $\varpi(H(d,2))$
 is 
 the morphism-colored functor $(\bar\eel,\ees_1)$
 defined by
  $\bar\eel(x)=\ast$ for $x\in \obj(\GGG(d,2))$,
  $\bar\eel((g,x))=\sum_{i=1}^d g_i$ for $(g,x)\in\mor(\GGG(d,2))$,
and $\ees_1(a)=a$ mod $2$ for $a\in\Set{0,1,\ldots,d}$.
 In this case,
 any morphism-colored functor from $H(d,2)$ 
 to any discrete morphism-colored category
 factors through  the morphism-colored functor $(\bar\eel,\ees)$
 to 
 the discrete morphism-colored category
 $(\Set{ \ast }/\!/(\ZZ/2\ZZ),\id)$.
 Hence we have the following proposition,
 which showed in Kuribayashi--Momose \cite{arxiv:1507.01745}:
 \begin{proposition}
  \label{prop:43}
  Fix a category $\CCC$.
For  positive integers  $d,d'$,
  we have a one-to-one correspondence between 
  morphism-colored functors from $H(d,2)$ to $(\CCC,\id)$ 
  and
  morphism-colored functors from $H(d',2)$ to $(\CCC,\id)$.
\end{proposition}

\end{document}